\documentclass{article}
\usepackage[top=3cm, bottom=3cm, left=3cm, right=3cm]{geometry}
\usepackage{amsmath}
\usepackage{amsthm}
\usepackage{enumerate}
\newtheorem{definition}{Definition}
\newtheorem{remark}[definition]{Remark}
\newtheorem{example}[definition]{Example}
\newtheorem{theorem}{Theorem}
\newtheorem{conjecture}[theorem]{Conjecture}
\newtheorem{lemma}[theorem]{Lemma}
\newtheorem{corollary}[theorem]{Corollary}

\newcommand{\bandwidth}[1]{\ensuremath{\mathrm{band}\!\left(#1\right)}}
\newcommand{\reducemat}[1]{\ensuremath{\mathrm{RED}\!\left(#1\right)}}
\newcommand{\columnone}[2]{\ensuremath{\mathrm{col}_{#1}\!\left(#2\right)}}
\newcommand{\insertfix}[2]{\ensuremath{\mathrm{INS}_{#1}\!\left(#2\right)}}
\newcommand{\deletefix}[2]{\ensuremath{\mathrm{DEL}_{#1}\!\left(#2\right)}}
\newcommand{\essence}[1]{\ensuremath{\mathrm{ESS}\!\left(#1\right)}}
\newcommand{\tracemat}[1]{\ensuremath{\mathrm{tr}\!\left(#1\right)}}
\newcommand{\oomat}[1]{\ensuremath{\mathcal{OOM}\!\left(#1\right)}}

\newcommand{\numfactors}[2][]{\ensuremath{\mathrm{fact}_{#1}\!\left(#2\right)}}
\newcommand{\invertrows}[2]{\ensuremath{\mathrm{INV}_{#1}\!\left(#2\right)}}
\newsavebox{\rowvector}
\savebox{\rowvector}[1.6em]{row\hspace{-1.6em}\raisebox{1ex}{$\longrightarrow$}}
\newcommand{\rowvec}{\usebox{\rowvector}}
\newcommand{\rowindexed}[2]{\ensuremath{\rowvec_{#1}\!\left(#2\right)}}
\newsavebox{\colvector}
\savebox{\colvector}[1.1em]{col\hspace{-1.1em}\raisebox{1ex}{$\rightarrow$}}
\newcommand{\colvec}{\usebox{\colvector}}
\newcommand{\colindexed}[2]{\ensuremath{\colvec_{#1}\!\left(#2\right)}}
\newcommand{\unsignrow}[2]{\ensuremath{\mathrm{FIX}_{#1}\!\left(#2\right)}}
\newsavebox{\distvector}
\savebox{\distvector}[1.7em]{dist\hspace{-1.6em}\raisebox{1ex}{$\longrightarrow$}}
\newcommand{\distvec}{\usebox{\distvector}}
\newcommand{\disttable}[1]{\ensuremath{\distvec\!\left(#1\right)}}
\newcommand{\norm}[2][]{\ensuremath{\left\|#2\right\|_{#1}}}
\newcommand{\greedymat}[1]{\ensuremath{\mathcal{GBM}\!\left(#1\right)}}
\newcommand{\opfactors}[2][]{\ensuremath{\overline{\mathrm{fact}}_{#1}\!\left(#2\right)}}
	\title{Factoring Permutation Matrices Into a Product of Tridiagonal Matrices}
	\author{Michael Daniel Samson and Martianus Frederic Ezerman\\\normalsize{\texttt{s080067@mail.ntu.edu.sg, mart0005@mail.ntu.edu.sg}}}
\begin{document}
\allowdisplaybreaks
\maketitle
	\begin{abstract}
		Gilbert Strang posited \cite{Strang} that a permutation matrix of bandwidth $w$ can be written as a product of $N < 2w$ permutation matrices of bandwidth 1.  A proof employing a greedy ``parallel bubblesort'' algorithm on the rows of the permutation matrix is detailed and further points of interest are elaborated.
	\end{abstract}

\section{Conjecture and Outline}\label{conjover}
This section states the problem, starting with some necessary definitions.  As a convention, $M'$ will denote the transpose of a matrix $M$ and $I$ is the identity matrix.
\begin{definition}\label{permmatrix}
	An $n \times n$ \emph{permutation matrix} $P$ contains only 0s and 1s, with only one 1 per row and column.  The permutation of a column vector $\vec{x}$, where $\vec{x}' = [1\ 2\ \cdots\ n]$, is the column vector $P\vec{x}$.

	A matrix $M = [m_{ij}]$ is said to be \emph{of bandwidth $w$}, denoted by $\bandwidth{P} := w$, if $m_{ij} = 0$ whenever $|i - j| > w$.
\end{definition}
The value of $\bandwidth{M}$ is 0 if $M$ is diagonal, 1 if $M$ is tridiagonal and 2 if $M$ is pentadiagonal.

Gilbert Strang posed the following conjecture \cite{Strang}
:  
\begin{conjecture}[Strang]\label{strangconjmat}
	A finite permutation matrix of bandwidth $w > 0$ can be written as the product of at most $2w - 1$ bandwidth--1 permutation matrices.
\end{conjecture}
\begin{example}\label{examplep}
\begin{displaymath}
	P = 
	\begin{bmatrix}
		0 & 0 & 1 & 0\\
		1 & 0 & 0 & 0\\
		0 & 0 & 0 & 1\\
		0 & 1 & 0 & 0
	\end{bmatrix}
	=
	\begin{bmatrix}
		0 & 1 & 0 & 0\\
		1 & 0 & 0 & 0\\
		0 & 0 & 0 & 1\\
		0 & 0 & 1 & 0
	\end{bmatrix}
	\begin{bmatrix}
		1 & 0 & 0 & 0\\
		0 & 0 & 1 & 0\\
		0 & 1 & 0 & 0\\
		0 & 0 & 0 & 1
	\end{bmatrix} = R_1R_2.
\end{displaymath}
\end{example}

This paper aims to prove Conjecture~\ref{strangconjmat} and explore topics opened during the development of this proof.  The next three sections (Sections~\ref{canonperms},~\ref{overtake}, and \ref{proveconj}) 
will cover the proof of the conjecture, outlining a greedy ``parallel bubblesort''\cite{Strang} strategy to determine a factor per iteration,
starting from a specific key class of permutation matrices through progressively larger classes of permutation matrices.  Section~\ref{furtherpoints} concludes the paper with some points of further interest.

\section{Strang Canonical Matrices}\label{canonperms}
This section performs three tasks.  First, it outlines the general ``parallel bubblesort'' algorithm.  Second, it introduces two classes of permutation matrices: Strang canonical matrices and settled matrices.  Third, Theorem~\ref{solidrepthm}, the main result of this section, establishes that Conjecture~\ref{strangconjmat} holds for tracefree settled Strang canonical matrices with so-called reducing matrix factors.
\begin{definition}
	If $P$ is an $n \times n$ permutation matrix, then the $m$th row of $P$, $1 \leq m \leq n$ is $\rowindexed{m}{P}$, the $m$th column of $P$ is $\colindexed{m}{P}$ and $[\columnone{1}{P}\ \cdots\ \columnone{n}{P}] := (P\vec{x})'$ where $\vec{x}' = [1\ \cdots\ n]$.  
	
	If $1 \leq i < j \leq n$, $\rowindexed{i}{P}$ and $\rowindexed{j}{P}$ are an \emph{inverted pair} if $\columnone{i}{P} > \columnone{j}{P}$ and are a \emph{contented pair} 
	otherwise.
	
	If $P = \prod\limits_{k = 1}^m T_k$, where $T_k$ are bandwidth--1 permutation matrices, then $\numfactors{P} := m$.
\end{definition}
If an $n \times n$ permutation matrix $P$ represents a permutation $\sigma$, then $\columnone{i}{P} = \sigma(i)$ for $1 \leq i \leq n$, and each inverted pair of $P$ represents an inversion of $\sigma$.  Also, $\bandwidth{P} = \max_{1 \leq i \leq n}|\columnone{i}{P} - i|$.
\begin{remark}\label{existproduct}
	The bubblesort algorithm \cite[p.~40]{Algo} shows that each permutation $\sigma$ is the product of transposition of adjacent elements, $\sigma = \prod \tau_i$.  Let $P_\rho$ be the permutation matrix representing $\rho$.  Then $P_{\tau_i}$ are all bandwidth--1 permutation matrices, and $P_\sigma = \prod P_{\tau_i}$.  Thus, a permutation matrix can always be written as a product of bandwidth--1 permutation matrices.
	
	Let $P = \prod\limits_{k = 1}^m T_k$ be a permutation matrix where $T_k$ are bandwidth--1 permutation matrices.  Since permutation matrices are unitary and bandwidth--1 permutation matrices are inversions, $${T_k}^2 = I \text{ implies } T_k' = {T_k}^{-1} = T_k \text{ and } P^{-1} = \left(\prod_{k = 1}^m T_k\right)^{-1} = \prod_{k = 1}^m {T_{m + 1 - k}}^{-1} = \prod_{k = 1}^m T_{m + 1 - k} = P'.$$  Thus, $\numfactors{P} = \numfactors{P^{-1}}$.  Whenever such a product is defined, denote the indexed matrices $P_0 = P$ and $P_k = T_kP_{k - 1} = \left(\prod\limits_{i = 1}^k T_{k + 1 - i}\right)P$.
\end{remark}
\begin{lemma}\label{algoterm}
	If $P = P_0$ is a finite permutation matrix, and $P_k = B_kP_{k - 1}$ where $B_k$ is a nonidentity permutation matrix that performs swaps only on inverted pairs of $P_{k - 1}$, then there is a number $m$ such that $P_m = I$.
\end{lemma}
\begin{proof}
	Since $B_k$ makes some inverted pairs of $P_{k - 1}$ contented, the number of inverted pairs of $P_k$ is less than that of $P_{k - 1}$.  Since the number of inverted pairs of $P$ is finite, there must be a number $m$ such that $P_m = I$.
\end{proof}
When $B_k$ only swaps adjacent rows, Lemma~\ref{algoterm} describes ``parallel bubblesorting'', as each bubblesort iteration reduces the number of inversions of a permutation, and $P = \prod\limits_{k = 1}^m B_k$ is the required decomposition.  The rest of the paper investigates the greedy selection operation 
$\{B_k\}$ to ensure that $m = \numfactors{P} < 2w$ where $w = \bandwidth{P}$.

Given $P = \prod\limits_{k = 1}^m B_k$ as in Lemma~\ref{algoterm}, for $0 \leq i, j \leq m$, $\rowindexed{m_i}{P_i} = \rowindexed{m_j}{P_j}$ if $\columnone{m_i}{P_i} = \columnone{m_j}{P_j}$.
\begin{definition}
	Treating $P$ as a block diagonal matrix with the finest partition, each diagonal matrix is called a \emph{section} of $P$.  A $1 \times 1$ section is \emph{trivial}.
	
	A row of a permutation matrix is said to be \emph{positive} (\emph{negative}, \emph{neutral}) if the 1 is to the right of (to the left of, on, respectively) the diagonal.  A column of a permutation matrix is said to be \emph{positive} (\emph{negative}, \emph{neutral}) if the 1 is above (below, on, respectively) the diagonal.  
\end{definition}
Each section $S_1, \dots, S_k$ of a permutation matrix $P$ is a permutation matrix, and $\numfactors{P} = \max\{\numfactors{S_1}, \dots, \numfactors{S_k}\}$.
\begin{figure}
	\begin{center}
		\begin{picture}(55,55)(-30,-30)
			\put(-20,-20){\makebox(0,0)[c]{$0$}}
			\put(-20,-10){\makebox(0,0)[c]{$0$}}
			\put(-20,0){\makebox(0,0)[c]{$1$}}
			\put(-20,10){\makebox(0,0)[c]{$0$}}
			\put(-10,-20){\makebox(0,0)[c]{$1$}}
			\put(-10,-10){\makebox(0,0)[c]{$0$}}
			\put(-10,0){\makebox(0,0)[c]{$0$}}
			\put(-10,10){\makebox(0,0)[c]{$0$}}
			\put(0,-20){\makebox(0,0)[c]{$0$}}
			\put(0,-10){\makebox(0,0)[c]{$0$}}
			\put(0,0){\makebox(0,0)[c]{$0$}}
			\put(0,10){\makebox(0,0)[c]{$1$}}
			\put(10,-20){\makebox(0,0)[c]{$0$}}
			\put(10,-10){\makebox(0,0)[c]{$1$}}
			\put(10,0){\makebox(0,0)[c]{$0$}}
			\put(10,10){\makebox(0,0)[c]{$0$}}
			\put(20,-20){\makebox(0,0)[c]{$-$}}
			\put(20,-10){\makebox(0,0)[c]{$+$}}
			\put(20,0){\makebox(0,0)[c]{$-$}}
			\put(20,10){\makebox(0,0)[c]{$+$}}
			\put(-20,20){\makebox(0,0)[c]{$-$}}
			\put(-10,20){\makebox(0,0)[c]{$-$}}
			\put(0,20){\makebox(0,0)[c]{$+$}}
			\put(10,20){\makebox(0,0)[c]{$+$}}
			\multiput(-25,-25)(40,0){2}{\line(0,1){40}}
			\multiput(-25,-25)(37,0){2}{\multiput(0,0)(0,40){2}{\line(1,0){3}}}
		\end{picture}
		\caption{Signs of rows and columns of $P$}
		\label{chargetargetex}
	\end{center}
\end{figure}
Each nontrivial section has a positive top row, a negative bottom row, a negative leftmost column and a positive rightmost column.  
\begin{remark}\label{signcount}
	Each 1 on the diagonal of a permutation matrix $P$ determines a neutral row and column.  Observe that the number of 1s in the upper triangle of $P$ is the sum of the number of its positive and neutral rows, and the sum of the number of its positive and neutral columns.  Thus, $P$ has the same number of positive rows and columns.  Observing the number of 1s in the lower triangle of $P$ similarly shows that $P$ has the same number of negative rows and columns.
\end{remark}
\begin{definition}
	Two permutation matrices are \emph{row-sign-equivalent} (\emph{column-sign-equivalent}) if the signs of their rows (columns) are the same.

	A section is \emph{row-settled} if all of its positive rows are above its negative rows, \emph{column-settled} if all of its negative columns are to the left of its positive columns and \emph{settled} if it is either row-settled or column-settled.  
	
	A row-settled (column-settled) matrix has only row-settled (column-settled) sections, and a settled matrix is either row-settled or column-settled.
	
	A section is \emph{upper-canonical} (\emph{lower-canonical}), or in \emph{upper-canonical form} (\emph{lower-canonical form}), if its positive (negative) rows are pairwise contented.  A section is \emph{Strang canonical} (\emph{half-canonical}), or in \emph{Strang canonical form} (\emph{half-canonical form}), if it is in both (either) upper-canonical and (or) lower-canonical form.
	
  A permutation matrix is upper-canonical (lower-canonical), or in upper-canonical form (in lower-canonical form) if all of its sections are in upper-canonical (lower-canonical) form.  A permutation matrix is Strang canonical (half-canonical), or in Strang canonical (half-canonical) form, if it is in both (either) upper-canonical and (or) lower-canonical form.
\end{definition}
The inverse of a row-settled matrix is a column-settled matrix, and vice-versa.  The inverse of an upper-canonical matrix is a lower-canonical matrix, and vice versa.  

If $U$ is an upper-canonical matrix and the 1s in its upper triangle excluding its diagonal are in rows $p_1, \dots, p_k$ such that $p_i < p_{i + 1}$, then $\columnone{p_i}{U} < \columnone{p_{i + 1}}{U}$ for $1 \leq i < k$.  If $L$ is a lower-canonical matrix and the 1s in its lower triangle excluding its diagonal are in rows $n_1, \dots, n_\ell$ such that $n_i < n_{i + 1}$, then $\columnone{n_i}{L} < \columnone{n_{i + 1}}{L}$ for $1 \leq i < \ell$.
\begin{example}
	The only $n \times n$, $n \leq 4$, sections that are not in Strang canonical form are
	\begin{displaymath}
		\begin{bmatrix}
			0 & 0 & 1 & 0\\
			0 & 0 & 0 & 1\\
			0 & 1 & 0 & 0\\
			1 & 0 & 0 & 0
		\end{bmatrix},
		\begin{bmatrix}
			0 & 0 & 0 & 1\\
			0 & 0 & 1 & 0\\
			1 & 0 & 0 & 0\\
			0 & 1 & 0 & 0
		\end{bmatrix} \text{ and }
		\begin{bmatrix}
			0 & 0 & 0 & 1\\
			0 & 0 & 1 & 0\\
			0 & 1 & 0 & 0\\
			1 & 0 & 0 & 0
		\end{bmatrix}.
	\end{displaymath}
	The eight $4 \times 4$ tracefree settled sections are the above three matrices and
	\begin{displaymath}
		\begin{bmatrix}
			0 & 1 & 0 & 0\\
			0 & 0 & 1 & 0\\
			0 & 0 & 0 & 1\\
			1 & 0 & 0 & 0
		\end{bmatrix},
		\begin{bmatrix}
			0 & 0 & 1 & 0\\
			0 & 0 & 0 & 1\\
			1 & 0 & 0 & 0\\
			0 & 1 & 0 & 0
		\end{bmatrix},
		\begin{bmatrix}
			0 & 1 & 0 & 0\\
			0 & 0 & 0 & 1\\
			1 & 0 & 0 & 0\\
			0 & 0 & 1 & 0
		\end{bmatrix},
		\begin{bmatrix}
			0 & 0 & 1 & 0\\
			1 & 0 & 0 & 0\\
			0 & 0 & 0 & 1\\
			0 & 1 & 0 & 0
		\end{bmatrix} \text{ and }
		\begin{bmatrix}
			0 & 0 & 0 & 1\\
			1 & 0 & 0 & 0\\
			0 & 1 & 0 & 0\\
			0 & 0 & 1 & 0
		\end{bmatrix}.
	\end{displaymath}
\end{example}
\begin{remark}\label{uniqcanonsign}
	A Strang canonical matrix is uniquely determined by the signs of its rows and columns.  The $m$th 1 on its diagonal is at the intersection of its $m$th neutral row and $m$th neutral column.  The $m$th 1 in its upper triangle, excluding the diagonal, is at the intersection of its $m$th positive row and $m$th positive column.  The $m$th 1 in its lower triangle, excluding the diagonal, is at the intersection of its $m$th negative row and $m$th negative column.
\end{remark}
\begin{definition}
	A \emph{reducing swap} $\hat{R}$ of a permutation matrix $P$ is an elementary matrix that swaps adjacent rows of $P$ where the upper row is positive and the lower row is negative.  The pair of rows are \emph{reduced} by the swap.  The \emph{reducing matrix} $R$ of permutation matrix $P$ is $R = \prod \hat{R}$ over all possible reducing swaps $\hat{R}$ of $P$, and the \emph{reduction} of $P$ is $\reducemat{P} := RP$.
\end{definition}
Each reducing swap makes an inverted pair contented, and if $R$ is the reducing matrix of a nonidentity permutation matrix, $\bandwidth{R} = 1$.
\begin{example}
	In Example~\ref{examplep}, $P = R_1R_2$, where $R_1$ is the reducing matrix of $P$ and $\reducemat{P} = P_1 = R_1P = R_2$.
\end{example}
\begin{lemma}\label{canonderangeredlemma}
	If $P$ is a Strang canonical matrix whose nontrivial sections are tracefree, then $P = \prod\limits_{k = 1}^m R_k$, where $R_k$ is the reducing matrix of $P_{k - 1}$.
\end{lemma}
\begin{proof}\label{canonderangeredproof}
	The permutation matrix $P$ and $\reducemat{P}$ have the same Strang canonicity, since reducing swaps only exchange the positions of a positive row and a negative row, so each $P_k$ is Strang canonical.
	
	Since the nontrivial sections of $P = P_0$ are tracefree, the only neutral rows of $P$ are in trivial sections.  Let $$P_k = \begin{bmatrix}A_k & 0 & B_k\\0 & 1 & 0\\C_k & 0 & D_k\end{bmatrix}, k > 0 \text{ and } P_{k - 1} = \begin{bmatrix}A_{k - 1} & B_{k - 1}\\C_{k - 1} & D_{k - 1}\end{bmatrix}.$$  The indicated neutral row of $P_k$ is in a trivial section if $B_k$ and $C_k$ are zero matrices.  Let $A_k$ be $(m - 1) \times (m - 1)$ and $\rowindexed{m}{P_{k - 1}}$ be signed.  
	
	If $\rowindexed{m}{P_{k - 1}}$ is negative, let $A_{k - 1}$ be $m \times m$.  Since $P_{k - 1}$ is Strang canonical, every negative $\rowindexed{m'}{P_{k - 1}}$ with $m' < m$ has its 1 in $A_{k - 1}$.  Then $\rowindexed{m - 1}{P_{k - 1}}$ is positive with $\columnone{m - 1}{P_{k - 1}} = m$ and every positive $\rowindexed{m'}{P_{k - 1}}$ with $m' < m$ has its 1 in $A_{k - 1}$.  Since each neutral $\rowindexed{m'}{P_{k - 1}}$ with $m' < m$ has its 1 in $A_{k - 1}$, both $B_{k - 1}$ and $C_{k - 1}$ are zero matrices, so $B_k$ and $C_k$ are zero matrices.  
	
	If $\rowindexed{m}{P_{k - 1}}$ is positive, let $A_{k - 1}$ be $(m - 1) \times (m - 1)$, and it similarly follows that $B_{k - 1}$, $C_{k - 1}$, $B_k$ and $C_k$ are zero matrices.
	
	Thus, $P$ and $\reducemat{P}$ are Strang canonical matrices with tracefree nontrivial sections, and the conclusion follows from Lemma~\ref{algoterm}.
\end{proof}
\begin{theorem}\label{solidrepthm}
	A tracefree settled Strang canonical matrix of bandwidth $w$ can be written as the product of less than $2w$ bandwidth--1 matrices.
\end{theorem}
\begin{proof}
	Let $P$ be a row-settled Strang canonical matrix with $\tracemat{P} = 0$ and $\bandwidth{P} = w$.  Thus, it has an $n \times n$ row-settled Strang canonical section $S$ with $\tracemat{S} = 0$ and $\bandwidth{S} = w$.  Let $\columnone{m}{S} = n$.  Then the upper $m$ rows of $S$ are positive, and the rest are negative, with $\columnone{m + 1}{S} = 1$.  From Lemma~\ref{canonderangeredlemma}, a row-settled Strang canonical section is the product of reducing matrices: 
	once a row is reduced, it is reduced by the next reducing matrix, until it is in a trivial section.  Hence, 
	\begin{equation}
		\numfactors{S} = \max_{1 \leq i \leq m} \{\columnone{i}{S} + m - 2 * i\} = \max_{m < i \leq n} \{2 * i - m - 1 - \columnone{i}{S}\}\label{settledcount}\tag{*}
	\end{equation}
	$|\columnone{i}{S} - i| \leq w$ indicates the \emph{swap count}, which is the number of reducing swaps for $\rowindexed{i}{S}$ to be placed in the $\columnone{i}{S}$th row, and $m - i$, if $i \leq m$, ($i - (m + 1)$, if $i > m$) is the \emph{delay count}, which is the number of positive (negative) rows that must be reduced before the positive (negative) $\rowindexed{i}{S}$ is first reduced.

	Since $\columnone{m}{S} = n$, $S$ has $n - m = \columnone{m}{S} - m \leq w$ negative rows; since $\columnone{m + 1}{S} = 1$, $S$ has $m = m + 1 - \columnone{m + 1}{S} \leq w$ positive rows.  Thus, if $\bandwidth{S} = w$, then $\numfactors{S} \leq w - 1 + w = 2w - 1$.
	
	Since a column-settled Strang canonical matrix $P$ is the inverse of a row-settled Strang canonical matrix $P^{-1}$ and, by Remark~\ref{existproduct}, $\numfactors{P} = \numfactors{P^{-1}}$, the conclusion follows.
\end{proof}
\begin{remark}\label{eventight}
	If $C$ is an $n \times n$ circulant nonidentity permutation matrix, then it is Strang canonical, row-settled and column-settled, with $\tracemat{C} = 0$ and $\numfactors{C} = n - 1$, all reducing matrices.  Moreover, for some $k$, $1 < k \leq n$, it follows that $\columnone{m}{C} = ((k + m - 2) \bmod n) + 1$, $$\bandwidth{C} = \max\{\columnone{1}{C} - 1, n - \columnone{n}{C}\} = \left\{\begin{array}{cl}k - 1, & \text{ if } k - 1 \geq n / 2\\n - k + 1, & \text{ if } k - 1 \leq n / 2\end{array}\right.,$$ $\colindexed{k}{C}$ is its first positive column and $\rowindexed{n - k + 2}{C}$ is its first negative row.  $\numfactors{C}$ is tight with Strang's bound if $n = 2w$ and $w = \bandwidth{C}$.
\end{remark}
\begin{corollary}\label{circbound}
	Let $C$ be a circulant matrix, $R$ be a row-settled Strang canonical matrix which is row-sign-equivalent to $C$ and $P$ be a column-settled Strang canonical matrix which is column-sign-equivalent to $C$.  Then $\numfactors{C} \geq \numfactors{R}, \numfactors{P}$.
\end{corollary}
\begin{proof}
	From Theorem~\ref{solidrepthm}, for each row of $R$ and $C$ with the same index, the delay count is the same, but the swap count is maximum for $C$.  So, from Equation~\eqref{settledcount}, $\numfactors{C} \geq \numfactors{R}$.  Since $P^{-1}$ is row-settled and Strang canonical and $C^{-1}$ is circulant, $\numfactors{C} = \numfactors{C^{-1}} \geq \numfactors{P^{-1}} = \numfactors{P}$ by Remark~\ref{existproduct}.
\end{proof}

\section{Overtaking Swaps}\label{overtake}
Theorem~\ref{solidrepessthm}, the main result in this section, establishes that Conjecture~\ref{strangconjmat} holds for settled Strang canonical matrices by a comparison with tracefree matrices
.
\begin{remark}\label{notfixedonly}
	Since the neutral rows of a permutation matrix $P = P_0$ are pairwise contented, if $P_k = B_kP_{k - 1}$ as in Lemma~\ref{algoterm} and all the signed rows of $P$ are neutral in $P_m$, then $P_m = I$.
\end{remark}
\begin{definition}
	An \emph{overtaking swap} of a permutation matrix $P$ is an elementary matrix that swaps an adjacent inverted pair of $P$ where either the upper row is not positive or the lower row is not negative.  The upper positive row or the lower negative row \emph{overtakes} the other row by the swap.
	\begin{align*}
		\text{If } P = \begin{bmatrix}A & B\\C & D\end{bmatrix} & \text{ where }
	A\text{ is } (m - 1) \times (m - 1), \text{ define } \insertfix{m}{P} := \begin{bmatrix}A & 0 & B\\0 & 1 & 0\\C & 0 & D\end{bmatrix}\\
	& \text{ and } \insertfix{m_1, \dots, m_k}{P} := \insertfix{m_k}{\cdots\left(\insertfix{m_1}{P}\right)}.\\
		\text{If } P = \begin{bmatrix}A & 0 & B\\0 & 1 & 0\\C & 0 & D\end{bmatrix} & \text{ where }
	A \text{ is } (m - 1) \times (m - 1), \text{ define } \deletefix{m}{P} := \begin{bmatrix}A & B\\C & D\end{bmatrix}\\
	& \text{ and } \deletefix{m_1, \dots, m_k}{P} := \deletefix{m_k}{\cdots\left(\deletefix{m_1}{P}\right)}.
	\end{align*}
	If $P$ is a nonidentity permutation matrix whose neutral rows have indices $r_1, \dots, r_k$, $r_i > r_{i + 1}$, $\essence{P} := \deletefix{r_1, \dots, r_k}{P}$ is the \emph{essential form} or \emph{essence} of $P$.
\end{definition}
\begin{remark}\label{fixedinserteffect}
	If $P$ is an $n \times n$ permutation matrix and $1 \leq m \leq n$, then $\bandwidth{\insertfix{m}{P}} \leq \bandwidth{P} + 1$.
\end{remark}
\begin{example}\label{transpex}
	To demonstrate the effect of inserting neutral rows on the number of factors, consider $T_1 = \begin{bmatrix}0 & 1\\1 & 0\end{bmatrix}$, the unique bandwidth--1 section.  The factorization of $T_k = \insertfix{2}{T_{k - 1}}$, with matrices containing overtaking swaps underlined
	\ is as follows:
	\begin{align*}
		T_2 & = \insertfix{2}{T_1} = 
		\begin{bmatrix}
			0 & 0 & 1\\
			0 & 1 & 0\\
			1 & 0 & 0
		\end{bmatrix} = 
		\underline{
		\begin{bmatrix}
			0 & 1 & 0\\
			1 & 0 & 0\\
			0 & 0 & 1
		\end{bmatrix}} 
		\begin{bmatrix}
			1 & 0 & 0\\
			0 & 0 & 1\\
			0 & 1 & 0
		\end{bmatrix}
		\begin{bmatrix}
			0 & 1 & 0\\
			1 & 0 & 0\\
			0 & 0 & 1
		\end{bmatrix};\\
		T_3 & = \insertfix{2}{T_2}
		=
		\begin{bmatrix}
			0 & 0 & 0 & 1\\
			0 & 1 & 0 & 0\\
			0 & 0 & 1 & 0\\
			1 & 0 & 0 & 0
		\end{bmatrix}
		 = 
		\underline{
		\begin{bmatrix}
			0 & 1 & 0 & 0\\
			1 & 0 & 0 & 0\\
			0 & 0 & 0 & 1\\
			0 & 0 & 1 & 0
		\end{bmatrix}} 
		\begin{bmatrix}
			1 & 0 & 0 & 0\\
			0 & 0 & 1 & 0\\
			0 & 1 & 0 & 0\\
			0 & 0 & 0 & 1
		\end{bmatrix} 
		\begin{bmatrix}
			0 & 1 & 0 & 0\\
			1 & 0 & 0 & 0\\
			0 & 0 & 0 & 1\\
			0 & 0 & 1 & 0
		\end{bmatrix};\text{ and }\\
		T_4 & = \insertfix{2}{T_3} = 
		\begin{bmatrix}
			0 & 0 & 0 & 0 & 1\\
			0 & 1 & 0 & 0 & 0\\
			0 & 0 & 1 & 0 & 0\\
			0 & 0 & 0 & 1 & 0\\
			1 & 0 & 0 & 0 & 0
		\end{bmatrix}
		= \underline{O_1O_2}R_1R_2R_3,
	\end{align*}
	where $O_1$ swaps the first and last pairs of rows of $T_4$ and $O_2$ swaps the middle row of $O_1T_4$.
\end{example} 
When a neutral row is overtaken, it assumes the sign of the overtaking row.
\begin{theorem}\label{solidrepessthm}
	A settled Strang canonical matrix of bandwidth $w$ can be written as the product of less than $2w$ bandwidth--1 matrices.
\end{theorem}
\begin{proof}
	To show the result for a settled Strang canonical matrix $P$, a circulant matrix $C$ will be used to determine an upper bound for $\numfactors{P}$.

	Let $C$ be $n \times n$ with $\colindexed{c}{C}$ as its first positive column and $\rowindexed{r}{C}$ as its first negative row.  From Remark~\ref{eventight}, $C = \prod\limits_{k = 1}^{n - 1} R_k$ where $R_k$ is the reducing matrix of $C_{k - 1}$.
		
	For $P = \insertfix{m}{C}$, let $w = \bandwidth{P} = \bandwidth{C} + 1$.  If $m = r$, the neutral row is inserted between the positive and negative rows.  The initial reduction $C_1$ is delayed by an overtake of the neutral row by $\rowindexed{m'}{P}$.  If $2m < n$, the neutral row is closer to $\rowindexed{1}{P}$.  If $2m > n$, it is closer to $\rowindexed{n + 1}{P}$.  So the neutral row is overtaken toward whichever row between $\rowindexed{1}{P}$ and $\rowindexed{n + 1}{P}$ it is closer to, or either if $2m = n$, and $\rowindexed{m'}{P}$ has the sign of $n - 2m$.  Let the overtaking swap be $O$ and $OP$ replace $P$.
	\begin{enumerate}[\text{Step} 1]
		\item If $\rowindexed{r'}{P}$ is first negative row of $P$, then $r' \in \{r, r + 1\}$.  Let $k = \min\{r' - 1, n + 1 - r'\}$ and $e = |r' - k| \in \{1, n + 1\}$.  For the matrices $O_q$ with $1 \leq q \leq k$, if $\columnone{m_q}{P_{q - 1}} = m$, then $O_q = \insertfix{m_q}{R_q}\hat{O}_q$ with $\hat{O}_q = I$ if $P_{q - 1}$ is Strang canonical or $\hat{O}_q$ is the overtaking swap of $\rowindexed{m_q}{P_{q - 1}}$ otherwise.
		\item $P_k$ is Strang canonical and $\rowindexed{e}{P_k}$ is neutral.  For $q > k$, $O_q$ can be determined by showing that $P_k = \insertfix{e}{\hat{C}_{k - 1}}$ where:
		\begin{enumerate}[\text{Case} 1]
			\item If $\rowindexed{m}{P_k}$ is neutral, then $P_k = \insertfix{e, e}{\bar{C}_{k - 1}}$, so $\hat{C} = \insertfix{e}{\bar{C}}$, where 
$\bar{C}$ is an $n - 1 \times n - 1$ circulant matrix
.
			\item If $m$ is between $c$ and $r$ or $m = c$, then $P_k = \insertfix{e}{C_{k - 1}}$, so $\hat{C} = C$.
			\item Otherwise, 
$\hat{C}$ is a row-settled Strang canonical matrix which is row-sign-equivalent to $C$
.  
%
		\end{enumerate}
		Then, for $k \leq q \leq \numfactors{\hat{C}}$, $O_{q + 1} = \insertfix{e}{\hat{R}_q}$, where $\hat{R}_q$ is the reducing matrix of $\hat{C}_{q - 1}$ and, from Corollary~\ref{circbound}, $\numfactors{P} = \numfactors{\hat{C}} + 1 \leq n + 1$.
	\end{enumerate}
	Upon the completion of the above steps, it can be determined that $P = \prod\limits_{k = 1}^q \bar{O}_k$ where $\bandwidth{\bar{O}_k} = 1$.  If $m = r = c$, then $n = 2(m - 1)$ and $\bandwidth{C} = m - 1$.  Thus, $\bandwidth{P} = m = w$, $q = n + 1 = 2w - 1$, $\bar{O}_1 = O$ and $\bar{O}_{k + 1} = O_k$ for $1 \leq q \leq n$.  Otherwise, $q \in \{n - 1, n\}$ and, by Remark~\ref{eventight}, $2(w - 1) - 1 \geq n - 1$ making $2w - 1 \geq n + 1 > q$.  Therefore, for $P = \insertfix{m}{C}$, $\numfactors{P} < 2w$, and this bound is tight only when $m = \frac{n}{2} + 1$.  
	
	As seen in Example~\ref{transpex}, the parity of the number of neutral rows inserted as a block, say $P = \insertfix{m, \dots, m}{C}$, whether it is an odd or an even number, may affect $\numfactors{P}$ differently.  In particular, when $m = \frac{n}{2} + 1$, $\numfactors{\insertfix{m, m}{C}} = \numfactors{\insertfix{m}{C}}$ and $\bandwidth{\insertfix{m, m}{C}} = \bandwidth{\insertfix{m}{C}} + 1$.  For $m_1, \dots, m_\ell$ such that $1 < m_{i + 1} < m_i \leq n$, $$\numfactors{\insertfix{m_1, \dots, m_\ell}{C}} - \numfactors{C} = \sum\limits_{i = 1}^\ell \left(\numfactors{\insertfix{m_i}{C}} - \numfactors{C}\right).$$  

	Multiple blocks of neutral rows inserted to produce $P$ occasionally add a single bandwidth--1 factor, whenever $P_q$ contains a neutral row between rows of the opposite sign, such as, if $r' > r$, for $\insertfix{r, r + 2, r + 2, r + 2}{C}$ and for $\insertfix{r' - r, \dots, r' - r}{C}$ where there are $r + 1$ neutral rows inserted.  
		
	Therefore, for any circulant nonidentity permutation matrix $C$, given the class $\mathcal{C}_C = \{P : \essence{P} = C\}$, if 
	$d_P = 2\bandwidth{P} - 1 - \numfactors{P}$, then $d_C \geq 0$, $d_C = d_P$ when $C$ is a $2m \times 2m$ matrix and $P = \insertfix{m + 1}{C}$, otherwise $d_P > d_C$ whenever $P \neq C$,.

	Finally, if $R$ is a row-settled Strang canonical matrix which is row-sign-equivalent to $C$, from Corollary~\ref{circbound} and by following the previous arguments, for every set $\{m_1, \dots, m_t\}$, $\bandwidth{\insertfix{m_1, \dots, m_t}{R}} = \bandwidth{\insertfix{m_1, \dots, m_t}{C}}$ and $\numfactors{\insertfix{m_1, \dots, m_t}{R}} \leq \numfactors{\insertfix{m_1, \dots, m_t}{C}}$.  The argument holds for column-settled Strang canonical $R^{-1}$, and the conclusion follows.
\end{proof}
\begin{remark}\label{fixedlowbound}
	If $P$ is a settled Strang canonical matrix with $\bandwidth{P} = w$ and $f$ neutral rows, then $\numfactors{P} < 2w - f$, by the proof of Theorem~\ref{solidrepessthm}, noting the tight-bound exception.
\end{remark}

\section{Opportunistic Overtaking}\label{proveconj}
The main result of this section is the completion of the proof of Conjecture~\ref{strangconjmat} with opportunistic-overtaking matrices---a greedy generalization of reducing matrices---along the construction used in Theorem~\ref{solidrepessthm}
.
\begin{definition}
	If $P$ is a permutation matrix, then $\invertrows{m}{P}$ is the minimal submatrix containing only consecutive rows of $P$ such that $\rowindexed{m}{P}$ and all of the rows of $P$ that are pairwise inverted with $\rowindexed{m}{P}$ are in $\invertrows{m}{P}$.

	An \emph{inverted block} of a permutation matrix is a maximal submatrix containing only consecutive rows such that all the rows are pairwise inverted.
	
	An \emph{opportunistic-overtaking matrix} $O$ of a permutation matrix $P$ is the product of the reducing matrix of $P$ and the overtaking swaps of $P$ such that, for every inverted block of $P$, the only rows that $O$ can leave unswapped are the first and the last rows of the block.  The collection of all products of $P$ with any of its opportunistic-overtaking matrices $O$ is denoted by $\oomat{P} \ni OP$.

	Given a signed $\rowindexed{m}{P}$, \emph{removing its sign} produces a matrix $P' := \unsignrow{m}{P}$.  If $m'$ is such that $\columnone{m'}{P} = m$, then: $\rowindexed{k}{P'} = \rowindexed{k}{P}$ when $k \neq m, m'$; $\rowindexed{m'}{P'} = \rowindexed{m}{P}$; and $\rowindexed{m}{P'}$ is neutral.
\end{definition}	
For every $\rowindexed{m}{P}$ not in a trivial section, the top row of $\invertrows{m}{P}$ is positive and the bottom row of $\invertrows{m}{P}$ is negative.  If $P$ is upper-canonical and $\rowindexed{m}{P}$ is positive, it is the top row of $\invertrows{m}{P}$.  If $P$ is lower-canonical and $\rowindexed{m}{P}$ is negative, it is the bottom row of $\invertrows{m}{P}$.

If the rows of an $n \times n$ permutation matrix $P$, from $\rowindexed{i}{P}$ to $\rowindexed{j}{P}$, form an inverted block, then
\begin{itemize}
	\item either $i = 1$ or $\rowindexed{i - 1}{P}$ and $\rowindexed{i}{P}$ are a contented pair, and
	\item either $j = n$ or $\rowindexed{j}{P}$ and $\rowindexed{j + 1}{P}$ are a contented pair.
\end{itemize}

If $P$ is a nonidentity permutation matrix and $O$ is any of its opportunistic-overtaking matrices, then $\bandwidth{O} = 1$ and $O$ satisfies the ``parallel bubblesort'' condition of $B_k$ from Lemma~\ref{algoterm}, while providing a locally-optimal, i.e. greedy, condition to determine the next bandwidth--1 factor, in that $P$ and $OP \in \oomat{P}$ share no inverted pairs.

The only inverted blocks that a Strang canonical matrix has are reducible pairs and neutral rows with a positive and/or a negative row to overtake it.  The product of any permutation matrix and any of its opportunistic-overtaking matrices is a permutation matrix whose inverted blocks have no more than three rows.

A permutation matrix $P$ always has a unique reducing matrix.  $P$ has a unique opportunistic-overtaking matrix only if each inverted block of $P$ that has more than two rows has a reducible pair, otherwise that block can have two choices of overtaking swaps. 

An algorithm for determining an opportunistc-overtaking matrix of $P$ is given in the Appendix.
\begin{example}
	If, as in Theorem~\ref{solidrepthm} and the sections in Lemma~\ref{canonderangeredlemma}, $P$ is Strang canonical and $\tracemat{P} = 0$, then its reducible pairs are inverted blocks and $\oomat{P} = \{\reducemat{P}\}$.  In the proof of Theorem~\ref{solidrepessthm}, $P_q \in \oomat{P_{q - 1}}$
	.
\end{example}
\begin{theorem}\label{strangthm}
	A permutation matrix of bandwidth $w$ can be written as the product of less than $2w$ bandwidth--1 permutation matrices.
\end{theorem}
\begin{proof}[Proof of Conjecture~\ref{strangconjmat}]
	The proof will relax the conditions on the permutation matrix and prove that the conjecture holds for each relaxation.

	Let $P$ be a lower-canonical matrix with $\bandwidth{P} = w$ and $P = \prod\limits_{k = 1}^q O_k$ where $O_kP_{k - 1} = P_k \in \oomat{P_{k - 1}}$.  
A negative row of $P$ will be in a trivial section in some $P_k$ only by being swapped by $O_k$ with a positive row of $P$.  So, let $\rowindexed{m}{P}$ be positive.
	\begin{enumerate}[\text{Case} 1]
		\item If $\rowindexed{m}{P} = \rowindexed{m_k}{P_k}$ is never overtaken in $\{O_k\}$, the plan is to localize the determination of the swaps that move $\rowindexed{m}{P}$ in $\{O_k\}$.  
		
		First, determine $\tilde{P}^m$ such that, through removing the signs of the positive rows that are overtaken by $\rowindexed{m}{P}$.  $\tilde{P}^m$ has no such positive row.  
		
		Next, determine the matrix $P^m$ localizing to the swaps of $\rowindexed{m}{P}$ and the rows that are inverted with it.  
		Then, $P^m$ is column-settled and Strang canonical whose only nontrivial section is from $\rowindexed{m}{P^m}$ to $\rowindexed{\ell_m}{P^m}$, which is row-sign-equivalent to $\invertrows{m}{\tilde{P}^m}$.
		
		If, by Theorem~\ref{solidrepessthm}, $P^m = \prod\limits_{k = 1}^{\bar{q}_m} T^m_k$, with $\hat{T}^m_k$ the possibly identity elementary matrix performing the swap in $T^m_k$ that moves $\rowindexed{m_k}{P^m_k} = \rowindexed{m}{P^m}$, can be performed on $\rowindexed{m}{\tilde{P}^m}$, and thus can also be performed on $\rowindexed{m}{P}$.  
		\label{overtakecase}
		
		There are two scenarios to consider:
		\begin{enumerate}[\text{Sub-Case} 1]
			\item Assume that there is a reducible pair in $\invertrows{m}{P}$ that is not in $\invertrows{m}{\tilde{P}^m}$.  Since both rows have their signs removed, their reducing swap will be in $O_1$, and $\rowindexed{m}{P}$ will be in a trivial section in $P_{\bar{q}_m + 1}$.  By Remark~\ref{fixedlowbound}, $q_m = \bar{q}_m + 1 < 2w - f_m + 1$ with $f_m > 1$.
			\item Otherwise, $\rowindexed{m}{P}$ will be in a trivial section in $P_{\bar{q}_m}$.  By Remark~\ref{fixedlowbound}, $q_m = \bar{q}_m < 2w - f_m$.\label{canonicalalwayssubcase}
		\end{enumerate}
		\item Let the $m$th row be overtaken in $\{O_k\}$.  From the previous case
		, for each $\rowindexed{r'}{P}$ overtaking $\rowindexed{m}{P}$ in $\{O_k\}$, $\rowindexed{r'}{P}$ is in a trivial section in $P_{q_{r'}}$ with $q_{r'} < 2w - f_{r'}$ where $f_{r'}$ is the number of nonnegative rows in $\invertrows{r'}{P}$.  Again, there are two scenarios to consider:
		\begin{enumerate}[\text{Sub-Case} 1] 
			\item If the final swap of $\rowindexed{m}{P}$ in $\{O_k\}$ is with one of the rows overtaking it, say $\rowindexed{r'}{P}$, then $\rowindexed{m}{P}$ is in a trivial section in $P_{q_m}$, where $q_m \leq q_{r'} < 2w - f_{r'}$.  
			\item Otherwise, 
			$\rowindexed{m}{P}$ is positive just before it is in a trivial section, and all rows that can overtake it have overtaken it before it is swapped into a trivial section.  Thus, after the last row overtakes it in $P_k$, $\rowindexed{m}{P} = \rowindexed{m_k}{P_k}$ is above its overtaking row $\rowindexed{r_k}{P_k}$, and once $\rowindexed{r_k}{P_k}$ swaps with a row below it, $\rowindexed{m}{P}$ can swap with the row it was overtaken by unless it was first overtaken by $\rowindexed{r_k}{P_k}$ and is contented with $\rowindexed{m}{P}$.  Then $\rowindexed{m}{P}$ has a delay count trailing $\rowindexed{r_k}{P_k}$ of at most $f_{r_k}$ and $\rowindexed{m}{P}$ is in a trivial section in $P_{q_m}$, where $q_m \leq q_{r_k} + f_{r_k} < 2w$.
		\end{enumerate}
	\end{enumerate}
	Since $\numfactors{P} = q = \max_{\columnone{m}{P} \neq m} q_m$, then $q < 2w$ and the conjecture holds for lower-canonical matrices.  Since an upper-canonical matrix is the inverse of a lower-canonical matrix, the same conclusion follows from Remark~\ref{existproduct}.
	
	If $P$ is not half-canonical, $P^m$ can be replaced in Case~\ref{overtakecase} by a column-settled upper-canonical matrix, where the same negative rows of $P^m$ and $P$ constitute an inverted pair, and the results will similarly follow.
\end{proof}
If, in the above proof, $P$ is Strang canonical, then only Sub-Case~2
\ of Case~\ref{overtakecase} holds for each signed row $\rowindexed{m}{P}$.

\section{Further Points for Analysis}\label{furtherpoints}
	Panova \cite{Panova} proved Conjecture~\ref{strangconjmat} through the use of wiring diagrams.  This approach is similar to determining \emph{multi-braids}
	.  From braid theory, by the Artin relations \cite[Eq.~18, 19]{Braid}, braids that do not share a thread commute: here, any number of commuting braids can be combined, without ambiguity, into a single multi-braid.  It is of interest to compare the factors derived from the approach in \cite{Panova}, as with the method of Albert, Li and Yu \cite[Sec.~4]{Strang}, which is not yet readily available, with the opportunistic-overtaking matrices approach.
\begin{definition}
	The \emph{distance table} of an $n \times n$ permutation matrix $P$ is $\disttable{P} := (P - I)\vec{x}$ where $\vec{x}' = [1\ \dots\ n]$.
\end{definition}
\begin{remark}\label{totswapcount}
	The total number of reducing and overtaking swaps in the factors of a permutation matrix is the number of inverted pairs of that matrix.  This number is also half the sum of absolute values of the entries of its distance table, plus the number of rows that can overtake each signed row and half the number of rows that can overtake each neutral row.
	
	Given $P = \prod\limits_{k = 1}^m B_k$, as in Lemma~\ref{algoterm}, if the distance tables are taken as sequences, then, for the following standard norms,
	\begin{align*}
		                       \norm[\ell_1]{\disttable{P_k}} & \leq \norm[\ell_1]{\disttable{P_{k - 1}}},\\
		\bandwidth{P_k} = \norm[\ell_\infty]{\disttable{P_k}} & \leq \norm[\ell_\infty]{\disttable{P_{k - 1}}} = \bandwidth{P_{k - 1}},\\
		                       \norm[\ell_2]{\disttable{P_k}} & < \norm[\ell_2]{\disttable{P_{k - 1}}},
	\end{align*}
	indicating that the Manhattan and Chebychev distances cannot increase and that the Euclidean distance always decreases
	.
\end{remark}

The previous remark indicates that the subproducts $P_k$ of a given permutation matrix $P$ are ``diffusions'' of the initial state $\disttable{P}$, where a ``parallel bubblesort'' iteration is performed in each ``time-step''.  This may be better analyzed if a relevant basis can be found.

\begin{definition}
	A \emph{greedy bubble matrix} $G$ of a permutation matrix $P$ is a product of reducing and overtaking swaps of $P$ such that $P$ and $GP$ have no common inverted pairs.  The collection of all products of $P$ with any of its greedy bubble matrices $G$ is denoted by $\greedymat{P} \ni GP$.

	An \emph{optimal factorization} of a permutation matrix $P$ is $\prod\limits_{k = 1}^m T_k = P$ where $\bandwidth{T_k} = 1$ and each other factorization of $P$ into bandwidth--1 matrices cannot have less factors than $\opfactors{P} := m$.  
\end{definition}
$\oomat{P} \subseteq \greedymat{P}$, but a greedy bubble matrix of $P$ need not include reducing swaps of $P$.

A breadth-first spanning-tree algorithm \cite[Sec.~22.2]{Algo} rooted in the identity matrix applied to the Cayley graph of the symmetric group of length $n$, corresponding to set of $n \times n$ permutation matrices, 
whose connection set is the set of all permutations represented by bandwidth--1 matrices \cite{Cayley} can be used to determine $\opfactors{P}$ for any $n \times n$ permutation matrix $P$.  
\begin{remark}
	The number of $n \times n$ permutation matrices of bandwidth $w$, $w \leq 1$, is the $n$th Fibonacci number, $F_n$, where $F_0 = F_1 = 1$.
\end{remark}
In testing $n \times n$, $n \leq 9$, permutation matrices, the following were observed for every permutation matrix $P$: there is an optimal factorization $P = \prod\limits_{k = 1}^m G_k$, such that $P_k \in \greedymat{P_{k - 1}}$ and $\opfactors{P} \leq n$.

The former observation suggests that a greedy algorithm \cite[Ch.~16]{Algo} can determine $\opfactors{P}$;  Remark~\ref{totswapcount} indicates that the use of greedy bubble matrices is advantageous
.  Further observation leads to the following conjecture:

\begin{conjecture}\label{greedyconj}
	A finite permutation matrix of bandwidth $w > 0$ is the product of less than $2w$ greedy bubble matrices.
\end{conjecture}  
Conjecture~\ref{greedyconj} asserts that, if $P_0 = P$ and $P_k \in \greedymat{P_{k - 1}}$, then, for some $m < 2w$, $P_m = I$.

Of the tested greedy algorithms on $n \times n$ permutation matrices, $n \leq 9$, $\numfactors{P} \leq \opfactors{P} + \lfloor n / 3\rfloor$.

The latter observation seems provable from Theorem~\ref{strangthm} where, if $P$ is a permutation matrix, $\bandwidth{P} = w$, then for every signed $\rowindexed{m}{P}$, $\invertrows{m}{P}$ has at most $2w$ rows.  A.~M.~Bruckstein suggests using the sequence of adjacent transpositions to exhaustively generate all permutations of a given length, such the (Steinhaus-)Johnson-Trotter algorithm \cite{SJT}, as suggested in \cite{Even} or in \cite[Table~5]{Thesis}, and the Artin relations \cite{Braid}.

D.~Pasechnik suggests that the conjecture does not hold for infinite matrices. 

\section*{Acknowledgements}
The authors would like to thank Alfred Bruckstein and Fr\'{e}d\'{e}rique Oggier for suggesting the use of matrices instead of permutations and their supervision in the editing and reading of the drafts, Dmitrii Pasechnik and Radu Stancu for suggestions that were used in the permutation-oriented drafts, and Li-Lian Wang and Gilbert Strang for their advice and support.

\section*{Appendix: Opportunistic-Overtaking Matrix Algorithm}
Given: a permutation matrix $P$\\
Output: an opportunistic-overtaking matrix $O$ of $P$
\begin{itemize}
	\item Initialize $O = I$ and determine the inverted blocks of $P$, $B_1$, $B_2$, \dots, $B_k$
	\item For each inverted block of $P$, $B_i$, $1 \leq i \leq k$
	\begin{itemize}
		\item If $B_i$ has a reducible pair, $\rowindexed{m}{P}$ and $\rowindexed{m + 1}{P}$:
		\item[true:] While $\rowindexed{m - 2}{P}$ is in $B_i$, set $m$ to $m - 2$
		\item[false:] Let $\rowindexed{m}{P}$ be the top row of $B_i$
		\item While $\rowindexed{m + 1}{P}$ is in $B_i$, swap $\rowindexed{m}{O}$ and $\rowindexed{m + 1}{O}$, then set $m$ as $m + 2$
	\end{itemize}
\end{itemize}



\begin{thebibliography}{9}
	\bibitem{Cayley} B.~Alspach, ``Cayley Graphs'' in \underline{Topics in Algebraic Graph Theory}, \textit{Encyclopedia of Mathematics and Its Applications}
	, Vol.~102, pp.~156--178, Cambridge University Press, 2004.
	\bibitem{Braid} E.~Artin, \underline{The theory of braids}, \textit{Annals of Mathematics (2)}, Vol.~48, pp.~101--126, 1947.
	\bibitem{Algo} T.~Cormen, C.~Leiserson, R.~Rivest, C.~Stein, \textit{Introduction to Algorithms}, 3d ed, Massachusetts Institute of Technology, 2009.
	\bibitem{Even} S.~Even, \textit{Algorithmic Combinatorics}, Macmillan, 1973.
	\bibitem{Panova} G.~Panova, \underline{Factorization of Banded Permutations}, arXiV:1007.1760v1 [math.CO], 2010.
	\bibitem{Thesis} M.~D.~Samson, \textit{The Infinite Symmetric Group Part II: Nomenclature}, master's thesis, Ateneo de Manila University, 2008.
	\bibitem{Strang} G.~Strang, \underline{Fast transforms: Banded matrices with banded inverses}, \textit{Proceedings of the National Academy of Science}, Vol.~107, 
	pp.~12413--12416, 2010.
	\bibitem{SJT} H.~Trotter, \underline{Perm (Algorithm 115)}, \textit{Communications of the ACM}, Vol.~5, pp.~434--435, 1962.
\end{thebibliography}
\end{document}